\documentclass[a4paper, 11pt, final,underruledhead]{paper}
\usepackage{amssymb}
\usepackage{amsmath}
\usepackage{amsthm}
\usepackage{math}
\usepackage[mathscr]{euscript}
\usepackage{enumerate}
\usepackage[all]{xy}
\usepackage{graphics}
\usepackage[dvips]{graphicx}

\CompileMatrices

\newif\ifhozna \hoznafalse

\newif\ifcdn\cdnfalse

\begin{document}

%%%%%%%%%%%%%%%%%%%%%%%%%%%%%%%%%%%%%%%%%%%%%%%%%%%%%%%%%%%%%%%%%%%%%%%%%%%%%%%%%%%%
%%%%%%%%%%%%%%%%%%%%%%%%%%%%%%%%%%%%%%%%%%%%%%%%%%%%%%%%%%%%%%%%%%%%%%%%%%%%%%%%%%%%
%%%%%%%%%%%%%%%%%%%%%%%%%%%%%%%%%%%%%%%%%%%%%%%%%%%%%%%%%%%%%%%%%%%%%% PRIVATE TEX MACROS
%%%%%%%%%%%%%%%%%%%%%%%%%%%%%%%%%%%%%%%%%%%%%%%%%%%%%%%%%%%%%%%%%%%%%%%%%%%%%%%%%%%%
%%%%%%%%%%%%%%%%%%%%%%%%%%%%%%%%%%%%%%%%%%%%%%%%%%%%%%%%%%%%%%%%%%%%%%%%%%%%%%%%%%%%

%%%%%%%%% `FONTS'

\let\goth\mathfrak

%%%%%%%%% ENVIRONMENTS
%%%%%%%%%%%%%%%%% `theorems'

\iffalse
\newcommand\theoremname{Theorem}
\newcommand\lemmaname{Lemma}
\newcommand\corollaryname{Corollary}
\newcommand\propositionname{Proposition}
\newcommand\factname{Fact}
\newcommand\remarkname{Remark}
\newcommand\examplename{Example}

\newtheorem{thm}{\theoremname}[section]
\newtheorem{lem}[thm]{\lemmaname}
\newtheorem{cor}[thm]{\corollaryname}
\newtheorem{prop}[thm]{\propositionname}
\newtheorem{fact}[thm]{\factname}
\newtheorem{exmx}[thm]{\examplename}
\newenvironment{exm}{\begin{exmx}\normalfont}{\end{exmx}}

\newtheorem{rem}[thm]{\remarkname}
\fi

\def\myend{{}\hfill{\small$\bigcirc$}}

\newtheorem{reprx}[thm]{Representation}
\newenvironment{repr}{\begin{reprx}\normalfont}{\myend\end{reprx}}
\newtheorem{cnstrx}[thm]{Construction}
\newenvironment{constr}{\begin{cnstrx}\normalfont}{\myend\end{cnstrx}}
\def\classifname{Classification}
\newtheorem{classification}[thm]{\classifname}
\newenvironment{classif}{\begin{classification}\normalfont}{\myend\end{classification}}
\newcounter{facto}
\newenvironment{facto}[1]{%
\refstepcounter{facto}\begin{fact*}\space{\bf {#1}.\thefacto.}\space}{%
\end{fact*}}
%% \fi

\def\myend{{}\hfill{\small$\bigcirc$}}

\newenvironment{ctext}{%
  \par
  \smallskip
  \centering
}{%
 \par
 \smallskip
 \csname @endpetrue\endcsname
}

\iffalse %% ZAKRYTE LISTY
%%%%%%%%%%%%%%%%% `lists'

\newcounter{sentence}
\def\thesentence{\roman{sentence}}
\def\labelsentence{\upshape(\thesentence)}

\newenvironment{sentences}{%
        \list{\labelsentence}
          {\usecounter{sentence}\def\makelabel##1{\hss\llap{##1}}
            \topsep3pt\leftmargin0pt\itemindent40pt\labelsep8pt}%
  }{%
    \endlist}

\fi  %% ZAKRYTE LISTY

\newcounter{typek}\setcounter{typek}{0}
\def\thetypek{\roman{typek}}
\def\typitem#1{\refstepcounter{typek}\item[\normalfont(\roman{typek};\hspace{1ex}{#1})\quad]}
\newenvironment{typcription}{\begin{description}\itemsep-2pt\setcounter{typek}{0}}{% 
\end{description}}
\newcounter{ostitem}\setcounter{ostitem}{0}

%%%%%%%%%%%%%%%%%%%%%% MATH SYMBOLS

\iffalse
\newcommand*{\sub}{\raise.5ex\hbox{\ensuremath{\wp}}}
\newcommand{\nat}{{N}}   %%%{{\field N}}
\newcommand{\msub}{\mbox{\large$\goth y$}}    %%%{{\goth p}}   %%%{{\goth m}}
\def\id{\mathrm{id}}
\def\suport{{\mathrm{supp}}}
\def\Aut{\mathrm{Aut}}
\newcommand*{\struct}[1]{{\ensuremath{\langle #1 \rangle}}}
\fi

\def\id{\mathrm{id}}

\newcommand{\msub}{\mbox{\large$\goth y$}}    %%%{{\goth p}}   %%%{{\goth m}}
\def\suport{{\mathrm{supp}}}
\newcommand{\nat}{{N}}   %%%{{\field N}}
\def\bagnom#1#2{\genfrac{[}{]}{0pt}{}{#1}{#2}}

\def\kros(#1,#2){c_{\{ #1,#2 \}}}
\def\skros(#1,#2){{\{ #1,#2 \}}}
\def\LineOn(#1,#2){\overline{{#1},{#2}\rule{0em}{1,5ex}}}
\def\inc{\mathrel{\strut\rule{3pt}{0pt}\rule{1pt}{9pt}\rule{3pt}{0pt}\strut}}
\def\lines{{\cal L}}
\def\linessub{{\cal G}}
\def\kliki{{\cal K}}
\def\rozby{{\mathscr G}}
\def\tran{{\mathscr T}}
\def\collin{\sim}
\def\wspolin{{\bf L}}
\def\ncollin{\not\collin}

\def\VerSpace(#1,#2){{\bf V}_{{#2}}({#1})}
\def\GrasSpace(#1,#2){{\bf G}_{{#2}}({#1})}
\def\GrasSpacex(#1,#2){{\bf G}^\ast_{{#2}}({#1})}
\def\VeblSpSymb{{\bf P}\mkern-10mu{\bf B}}
\def\VeblSpace(#1){\VeblSpSymb({#1})}
\def\xwpis#1#2{{{\triangleright}\mkern-4mu{{\strut}_{{#1}}^{{#2}}}}}
%%%% TO CHOOSE THE BEST READABLE SYMBOL:  
% \def\MultVeblSymb{{\bf M}\mkern-14mu{\bf V}}        %%{\VeblSpSymb}
\def\MultVeblSymb{{\sf M}\mkern-10mu{\sf V}}        %%{\VeblSpSymb}
\def\MultVeblSymq{{\sf T}\mkern-6mu{\sf P}}        %%{\VeblSpSymb}
\def\xwlep(#1,#2,#3,#4){{\MultVeblSymb^{#1}{\xwpis{#2}{#3}}{#4}}}
\def\xwlepp(#1,#2,#3,#4,#5){{\MultVeblSymb_{#1}^{#2}{\xwpis{#3}{#4}}{#5}}}
\def\xwlepq(#1,#2,#3,#4,#5){{\MultVeblSymq_{#1}^{#2}{\xwpis{#3}{#4}}{#5}}}
\def\konftyp(#1,#2,#3,#4){\left( {#1}_{#2}\, {#3}_{#4} \right)}
%%\def\binokonf(#1){\konftyp({\binom{{#1}+1}{2}},{{#1}-1},\binom{{#1}+1}{3},3)}
%%\def\binokonf(#1){\konftyp({\binom{{#1}}{2}},{{#1}-2},\binom{{#1}}{3},3)}
%%\let\binoconf\binokonf
%% \iffalse
%% \newcounter{liczbaa}
%% \newcounter{liczbab}
%% \def\binkonf(#1,#2){\setcounter{liczbaa}#2 \setcounter{liczbab}#2
%%   \addtocounter{liczbaa}{1} \addtocounter{liczbab}{-2}
%%   \def\doa{\ifnum\csname c@liczbaa\endcsname =0\relax\else +\theliczbaa}
%%   \def\dob{\ifnum\csname c@liczbab\endcsname =0\relax\else +\theliczbab}
%%   \konftyp({\binom{{#1}\doa}{2}},{{#1}\dob},\binom{{#1}\doa}{3},3)  }
%% \fi
\newcount\liczbaa
\newcount\liczbab
\def\binkonf(#1,#2){\liczbaa=#2 \liczbab=#2 \advance\liczbab by -2
\def\doa{\ifnum\liczbaa = 0\relax \else
\ifnum\liczbaa < 0 \the\liczbaa \else +\the\liczbaa\fi\fi}
\def\dob{\ifnum\liczbab = 0\relax \else
\ifnum\liczbab < 0 \the\liczbab \else +\the\liczbab\fi\fi}
\konftyp(\binom{#1\doa}{2},#1\dob,\binom{#1\doa}{3},3) }
\let\binconf\binkonf

\def\perspace(#1,#2){\mbox{\boldmath$\Pi$}({#1},{#2})}
\def\vergras{{\goth R}}
\def\starof(#1){{\mathop{\mathrm S}(#1)}}
\def\topof(#1){{\mathop{\mathrm T}(#1)}}

%%%%%%%%%%%%%%%%%%%%%%%%%%%%%%%%%% SYMBOLS, NOTATIONS

\def\Ver{{\sf\bfseries VER}}
\def\Des{{\sf\bfseries DES}}
\def\Desv{{\sf\bfseries DES'}}
\def\Desr{{\sf\bfseries DES''}}
\def\MTC{{\sf STP}}
\def\MVC{{\sf MVC}}

\def\psts{partial Steiner triple system}
\def\PSTS{{\sf PSTS}}
\def\BSTS{{\sf BSTS}}
\def\asumpt{\rm({\sf P})}
\def\desAx{\rm({\sf Des})}

\def\brak{\relax} %{\warning[TO BE COMPLETED]}}

%% \begin{document}

\ifcdn
\else%%%%%%%%%%%%%%ZACZYNAMY DRUGI WARIANT
\bgroup

\title[$\konftyp(15,4,20,3)$-perspectives]{%
On a class of $\konftyp(15,4,20,3)$-configurations reflecting abstract properties of a %
perspective between tetrahedrons}
\author{M. Pra{\.z}mowska, K. Pra{\.z}mowski}

\maketitle 

\begin{abstract}
  In the paper we give a complete classification of schemes of abstract perspectives between two tetrahedrons
  such that intersecting edges corespond under this perspective to intersecting edges.
\end{abstract}

%%%%%%%%%%%%%%%%%%%%%%%%%%%%%%%%%%%%%%%%%%%%%%%%%%%%%%%%%%
%%%%%%%%%%%%%%%%%%%%%%%%%%%%%%%%%%%%%%%%%%%%%%%%%%%%%%%%%%
%%%%%%%%%%%%%%%%%%%%%%%%%%%%%%%%%%%%%%%%%%%%%%%%%%%%%%%%%% sec:intro
%%%%%%%%%%%%%%%%%%%%%%%%%%%%%%%%%%%%%%%%%%%%%%%%%%%%%%%%%%
%%%%%%%%%%%%%%%%%%%%%%%%%%%%%%%%%%%%%%%%%%%%%%%%%%%%%%%%%%
\section{Introduction}

The story begins with respectable Desargues Configuration and its generalizations. It is a common
opinion that the Desargues Configuration is/can be viewed as a schema of the perspective 
between two triangles in a projective space. In fact, there are three such schemes possible!
\par
In \cite{maszko} we have introduced a class of configurations representing schemas of generalized 
perspective between complete graphs, called {\em skew perspectives}.
Roughly speaking, such a perspective establishes a one-to-one correspondence
$\pi$ between the points of the graphs in question and a one-to-one correspondence 
$\xi$ between their edges. 
Some compatibility conditions on $\pi$ and $\xi$ are imposed.
In particular, if we require that $\xi$ preserves edge-intersection the variety of admissible skew 
perspectives seriously decreases (see \cite[Prop. \brak]{maszko}). 
Still, the number of such perspectives is, in general, huge, and no general tool 
suitable to classify them seems available.
\par
In this paper we give a complete classification of the
resulting configurations defined on $15$ points.
These are (generalized) perspectives between tetrahedrons.
Some of them can be also presented as a mutual perspective between three triangles
(defined and classified in \cite{STP3K5}).
Nevertheless, we obtain 62 new \psts s.

%%%%%%%%%%%%%%%%%%%%%%%%%%%%%%%%%%%%%%%%%%%%%%%%%%%%%%%%%%
%%%%%%%%%%%%%%%%%%%%%%%%%%%%%%%%%%%%%%%%%%%%%%%%%%%%%%%%%%
%%%%%%%%%%%%%%%%%%%%%%%%%%%%%%%%%%%%%%%%%%%%%%%%%%%%%%%%%% sec:defy
%%%%%%%%%%%%%%%%%%%%%%%%%%%%%%%%%%%%%%%%%%%%%%%%%%%%%%%%%%
%%%%%%%%%%%%%%%%%%%%%%%%%%%%%%%%%%%%%%%%%%%%%%%%%%%%%%%%%%
\section{Definitions}

Let us formulate the definition of a skew perspective configuration transformed from the original one
given in \cite{maszko} 
to the form suitable for the restricted case of $15$-points configurations.
So, let $I$ be a $4$-element set; in most parts we take 
$I = I_4 = \{ 1,2,3,4 \}$.
We write $\sub_k(I)$ for the set of $k$-subsets of $I$, and $S_I$ for the family of permutations
of $I$. The map 
$\varkappa\colon \sub_2(I) \ni a \longmapsto I\setminus a$ is a bijection of $\sub_2(I)$, 
called its {\em correlation}.
If $\varphi\in S_I$ we write $\overline{\varphi}$ for the natural extension of $\varphi$ to $\sub_2(I)$.
Let $Y\in\sub_3(I)$, we set $\topof(Y) = \sub_2(Y)$; if $i_0\in I$ we write 
$\starof(i_0) = \{ u\in\sub_2(I)\colon i_0\in u \}$ and $\topof(i_0) = \topof(I\setminus\{ i_0 \})$.
A set of the form $\starof(\strut)$ is called {\em a star}, and of the form $\topof(\strut)$ {\em a top}.

Let 
$A = \{ a_i\colon i\in I \}$ and $B = \{ b_i\colon i \in I\}$ be disjoint $4$-element sets, 
let $C = \{ c_u\colon u \in \sub_2(I) \}$ be a $6$-element set disjoint with
$A \cup B$, assume that $p\notin A \cup B \cup C$.
Let ${\goth N} = \struct{C,\lines}$ be a $\konftyp(6,2,4,3)$-configuration i.e.
let $\goth N$ be the Veblen-Pasch configuration labelled by the elements of $\sub_2(I)$.
Finally, let $\delta = \overline{\sigma}$ or
$\delta = \overline{\sigma}\circ\varkappa$ for a $\sigma\in S_{i_4}$.
Then we define the {\em skew perspective (of tetrahedrons)} 
with the {\em skew} $\delta$, the {\em center} $p$, and the {\em axis} $\goth N$
to be the structure
$$  
  \perspace(p,\delta,{\goth N}) \quad:=\quad \struct{\{ p \}\cup A \cup B \cup C, \lines_C\cup \lines_A \cup \lines_B%
  \lines_p},
$$
where
\begin{eqnarray*}
  \lines_A & = & \{ \{ a_i,a_j,c_{ \{i,j\} } \} \colon \{ i,j \}\in \sub_2(I_4) \},
  \\
  \lines_B & = & \{ \{ b_i,b_j,c_{\delta^{-1}(\{ i,j \})} \}\colon \{ i,j \}\in \sub_2(I_4)  \},
  \\
  \lines_p & = & \{  \{ p,a_i,b_i \}\colon i \in I_4.
\end{eqnarray*}
Then $\perspace(p,\delta,{\goth  N})$ is a $\konftyp(15,4,20,3)$-configuration. The term {\em perspective of graphs}
is justified by the following relations.
\begin{enumerate}[(a)]\itemsep-2pt
\item
  The sets $A$ and $B$ are complete $K_4$ graphs (of the collinearity relation).
\item 
  The lines which join the points which corespond each to other, $a_i$ and $b_i$ with $i\in I$, 
  meet in the center: the point $p$.
\item
  The lines which contain corresponding edges, 
  $a_ia_j$ and $b_{i'}b_{j'}$, $\{ i',j' \} = \delta( \{ i,j \} )$ with distinct $i,j\in I$, 
  meet in the axis: the configuration $\goth N$.
\end{enumerate}
Let us write $A^\ast = A \cup \{ p\}$ and $B^\ast = B\cup \{ p \}$.

All the labelings of $\sub_2(I_4)$ which yield a Veblen configuration
are known (cf. \cite[Fct. 2.5, Fig. 2]{klik:VC}).
This paper is not a right place to quote all the definitions that are needed
to define respective labelings. On the other hand, they can be recognized on the figures
\ref{fig:skos1}-\ref{fig:skos3}, namely: observe the lines which join the points $c_{i,j}$.
These are the structures named
\begin{equation}\label{veblist1}
  \GrasSpace(I_4,2) \text{ (Fig. \ref{fig:skos1})},
  \quad \VeblSpace(2) \text{ (Fig. \ref{fig:skos2})},
  \quad {\cal V}_5 \text{( Fig. \ref{fig:skos3})}.
\end{equation}
Moreover, the image of any of the three structures in the list \eqref{veblist1} under the map $\varkappa$ is 
again the Veblen configuration, to be denoted as follows, respectively.
\begin{equation}\label{veblist2}
  \varkappa(\GrasSpace(I_4,2)) = \GrasSpacex(I_4,2),\quad 
  \varkappa(\VeblSpace(2)) = {\cal V}_4, \quad
  \varkappa({\cal V}_5) = {\cal V}_6.
\end{equation}
In what follows we shall also frequently identify subsets of $C$ with
the corresponding subsets of $\sub_2(I)$:
$\topof(Y) = \{ c_u\colon u\in \topof(Y) \}$
and $\starof(i_0) = \{ c_u\colon u\in \starof(i_0) \}$.

\begin{fact}[{\cite[Fct. 2.5]{klik:VC}}]\label{fct:vebetyk}
%% It is known that 
  For every Veblen configuration ${\goth N}$ defined on $\sub_2(I_4)$ there is a 
  permutation $\alpha\in S_{I_4}$ such that 
  either $\overline{\alpha}$ or $\varkappa\overline{\alpha}$ 
  is an isomorphism of $\goth N$ onto $\goth V$,
  where $\goth V$ is a one 
  among \eqref{veblist1}. %% and \eqref{veblist2}.
%% of the defined in Figures \ref{fig:skos1}-\ref{fig:skos3}.
\end{fact}
In other words, the lists \eqref{veblist1} and \eqref{veblist2}
present all (up to a permutation of indices in $I_4$) the possible labellings 
of the Veblen configuration by the elements of $\sub_2(I_4)$.

\def\veblisty{\eqref{veblist1}{\&}\eqref{veblist2}}
In general, the `combinatorial' autmorphisms of the structures in  the lists \veblisty, quoted in 
\ref{fct:vebetyk} are also known.
%% \eqref{veblist1}{\&}\eqref{veblist2}
%
\begin{fact}\label{lem:vebauty}
  Let $\varphi \in S_{I_4}$ and let $\goth V$ be one in list \veblisty.
  Then $\varphi\in\Aut({\goth V})$ iff
  \begin{description}
  \item[{${\goth V} = \GrasSpace(I_4,2)$ or ${\goth V} = \GrasSpacex(I_4,2)$}:]
    $\varphi\in S_{I_4}$ is arbitrary.
  \item[{${\goth V} = \VeblSpace(2)$ or ${\goth V} = {\cal V}_4$}:]
    $\varphi \in S_{I_4}$ fixes two sets $\{1,2\}$ and $\{ 3,4\}$.
  \item[{${\goth V} = {\cal V}_5$ or ${\goth V} = {\cal V}_6$}:]
    $\varphi$ fixes an element $i_0\in I_4$ (e.g. $i_0 = 3$), so, in fact,
    $\varphi \in S_{I_4\setminus\{ 3 \}}$.
  \end{description}
\end{fact}
\begin{proof}[A sketch of \proofname]
 Note that each line of $\GrasSpace(I,2)$ has the form $\topof(Y)$ with $Y\in\sub_3(I)$
 and each line of $\GrasSpacex(I,2)$ has the form $\starof(i)$ with $i\in I$.
 In the reasoning below we consider only the structures in the list \eqref{veblist1}.
 The dual case (the list \eqref{veblist2}) runs analogously, with $\topof(\strut)$ replaced by $\starof(\strut)$.
  \begin{enumerate}[(a)]\itemsep-2pt
  \item\label{vebaut:typ1}
    Let ${\goth V} = \GrasSpace(I_4,2)$. Then, clearly, $\overline{\varphi}\in\Aut({\goth V})$
    for each $\varphi\in S_{I_4}$.
  \item\label{vebaut:typ2}
    Let ${\goth V} = \VeblSpace(2)$. It is known that $\goth V$ contains exactly two lines 
    of the form $\topof(Y)$, say these are 
    $\topof(\{ 1,3,4 \}) = \topof(2)$ and $\topof(\{ 2,3,4 \}) = \topof(1)$.
    Then $\varphi$ fixes the set $\{ 1,2 \}$ and $\{ 3,4 \}$. It is seen from the definition 
    of $\VeblSpace(2)$, that a permutation which fixes these two sets is an automorphism
    of $\goth V$.
  \item\label{vebaut:typ3}
    Let ${\goth V} = {\cal V}_5$. We see that $\goth V$ contains exactly one line of the
    form $\topof(Y)$; let it be $\topof(\{ 1,2,4 \}) = \topof(3)$. Consequently, $\varphi(3) = 3$.
    Again, it is seen that each permutation $\varphi\in S_{I_4\setminus\{ 3 \}}$
    extended by the condition $\varphi(3)=3$ determines the automorphism $\overline{\varphi}$ of $\goth V$.
  \end{enumerate}
\end{proof}

Let $\varphi,\psi\in S_{I_4}$,
we write $\varphi \sim \psi$ when $\varphi$ and $\psi$ are conjugate i.e.
when $\varphi = \psi^{\alpha} = \alpha \psi \alpha^{-1}$ for an $\alpha\in S_{I_4}$.
Set
$C(\varphi) = (x_1,\ldots,x_k)$ when $\varphi$ can be decomposed into disjoint cycles
$c_1,\ldots,c_k$ such that $x_i$ is the length of $c_i$, and $x_1\leq\ldots\leq x_k$.
Clearly, $\sum_{i=1}^k x_i = 4$.
It is a folklore that $\varphi \sim \psi$ is equivalent to $C(\varphi) = C(\psi)$.
\par\noindent
If $H\subset S_{I_4}$ then $\varphi \sim_H \psi$ denotes that
$\varphi = \psi^\alpha$ for an $\alpha \in H$.

It is a school exercise to determine, with the help of \ref{lem:vebauty} the following \ref{lem:vebtypy};
this determining may take some time and poor, though.
\begin{lem}\label{lem:vebtypy}
  Let $\goth V$ be a one of \veblisty. The following $\beta$'s are representatives
  of the equivalence classes of 
  $\sim_{\Aut({\goth V})}$ in the group $S_{I_4}$.
%  To complete our proof it sufices to make use of \ref{bool:klasif1} and 
%  consider the following three cases.
%%  In each of these cases every automorhism of $\goth V$ has form $\overline{\alpha}$
%  for an $\alpha\in S_{I_4}$.
  %
  \begin{description}\itemsep-2pt
  \item[{${\goth V} = \GrasSpace(I_4,2)$ or ${\goth V} = \GrasSpacex(I_4,2)$}:]\strut\linebreak
      $\beta = \id,\;\; (1)(2,3,4),\;\; (1,2)(3,4),\;\;
      (1)(2)(3,4),\;\; (1,2,3,4)$. 
  \item[{${\goth V} = \VeblSpace(2)$ or ${\goth V} = {\cal V}_4$}:]
%    Let us recall a characterization of $\Aut({\goth V})$.
%    It is known that $\goth V$ contains exactly two lines 
%    of the form $\topof(Y)$, say these are $\topof(\{ 1,3,4 \})$ and $\topof(\{ 2,3,4 \})$.
%    Then $\alpha$ fixes the sets $\{ 1,2 \}$ and $\{ 3,4 \}$. It is seen from the definition 
%    of $\VeblSpace(2)$, that a permutation which fixes these two sets is an automorphism
%    of $\goth N$.  
%    \par\noindent
%    So, let $\beta_2 = \beta_1^{\alpha}$ and $\overline{\alpha}\in \Aut({\goth V})$. 
%    From this, in particular, $\beta_2$ and $\beta_1$ are conjugate and therefore they
%    have the same number of cycles of equal length in the their cycle decomposition.
%    It is a school exercise to determine the orbits of $\Aut({\goth V})$ in $S_{I_4}$.
%    Representatives are as follows ($C(\beta)$ defined above \ref{fct:conjug}):
    %
    \begin{description}\itemsep-2pt
      \item[{$C(\beta) = (1,1,1,1)$}{\normalfont, then }] $\beta = \id$;
      \item[{$C(\beta) = (1,1,2)$}{\normalfont, then }] $\beta = (1)(2)(3,4),\;\; (1,2)(3)(4),\;\;(1)(2,3)(4)$;
      \item[{$C(\beta) = (1,3)$}{\normalfont, then }] $\beta = (1)(2,3,4),\;\; (4)(1,2,3)$;
      \item[{$C(\beta) = (2,2)$}{\normalfont, then }] $\beta = (1,2)(3,4),\;\; (1,4)(2,3)$;
      \item[{$C(\beta) = (4)$}{\normalfont, then }] $\beta = (1,2,3,4),\;\; (1,3,2,4)$.
    \end{description}
  \item[{${\goth V} = {\cal V}_5$ or ${\goth V} = {\cal V}_6$}:]
%    We see that $\goth V$ contains exactly one line of the
%    form $\topof(Y)$; let it be $\topof(\{ 1,2,4 \})$. Consequently, $\alpha(3) = 3$.
%    Again, it is seen that each permutation $\alpha\in S_{I_4\setminus\{ 3 \}}$
%    extended by the condition $\alpha(3)=3$ is an automorphism of $\goth V$.  
%    As above, we determine representatives of the orbits of $\Aut({\goth V})$ in $S_{I_4}$.
    \begin{description}\itemsep-2pt
      \item[$C(\beta) = (1,1,1,1)${\normalfont, then }] $\beta = \id$;
      \item[$C(\beta) = (1,1,2)${\normalfont, then }] $\beta = (1)(3)(2,4),\;\; (1)(2)(3,4)$;
      \item[$C(\beta) = (1,3)${\normalfont, then }] $\beta = (1)(2,3,4),\;\; (3)(1,2,4)$;
      \item[$C(\beta) = (2,2)${\normalfont, then }] $\beta = (1,2)(3,4)$;
      \item[$C(\beta) = (4)${\normalfont, then }] $\beta = (1,2,3,4)$.
    \end{description}
  \end{description}
\end{lem}

%%%%%%%%%%%%%%%%%%%%%%%%%%%%%%%%%%%%%%%%%%%%%%%%%%%%%%%%%%
%%%%%%%%%%%%%%%%%%%%%%%%%%%%%%%%%%%%%%%%%%%%%%%%%%%%%%%%%%
%%%%%%%%%%%%%%%%%%%%%%%%%%%%%%%%%%%%%%%%%%%%%%%%%%%%%%%%%% sec:permuty
%%%%%%%%%%%%%%%%%%%%%%%%%%%%%%%%%%%%%%%%%%%%%%%%%%%%%%%%%%
%%%%%%%%%%%%%%%%%%%%%%%%%%%%%%%%%%%%%%%%%%%%%%%%%%%%%%%%%%
\section{Perspectivities associated with permutations of indices}

In this section we concentrate upon the case when $\delta = \overline{\sigma}$, $\sigma\in S_{I_4}$;
we write, for short, $\delta = \sigma$.
In this case
\begin{ctext}
  $b_{\sigma(i)} \oplus b_{\sigma(j)} = c_{i,j}$ for distinct $i,j\in I_4$.
\end{ctext}
So, let us fix
\begin{equation}\label{typowe:1}
  {\goth M} = \perspace(p,\sigma,{\goth N}), \quad \sigma\in S_{I_4},\; 
  {\goth N} \text{ among listed in \veblisty}.
\end{equation}
It was already remarked that $\goth M$ (freely) contains at least two $K_5$ graphs:
$A^\ast$ and $B^\ast$.
As an important tool to classify our configurations let us quote after \cite[Prop. \brak]{maszko}
the following criterion
\begin{lem}\label{lem:nextgraf0}
  The following conditions are equivalent.
  \begin{sentences}\itemsep-2pt
  \item\label{lem2:cas1}
    $\perspace(p,\sigma,{\goth N})$ freely contains a complete 
    $K_{5}$-graph $G\neq A^\ast, B^\ast$.
  \item\label{lem2:cas2}
    There is $i_0 \in \Fix(\sigma)$ such that 
    $\starof(i_0) = \{ c_u\colon i_0\in u\in\sub_2(I) \}$
    is a collinearity clique in $\goth V$ freely contained in it.
  \end{sentences}
  In case \eqref{lem2:cas2},
  \begin{equation}\label{wzor:extragraf}
    G_{(i_0)} \; := \; \{ a_{i_0,b_{i_0}} \} \cup \starof(i_0)
  \end{equation}
  is a complete graph freely contained in $\perspace(p,\sigma,{\goth V})$.
\end{lem}
On the other hand all the $\konftyp(15,4,20,3)$-configurations which freely contain at least three $K_5$
were classified in \cite{STP3K5}.
So, our classification splits into two ways:
\begin{enumerate}[(A)]
\item
  $\Fix(\sigma)\ni i_0$ and $\starof(i_0)$ is a triangle in $\goth V$ for some $i_0\in I_4$
  \par\indent 
  \strut\quad- then $\goth M$ must be identified among those defined in \cite{STP3K5}. 
\item
  There is no $i_0\in I_4$ such that $\sigma(i_0) = i_0$ and $\starof(i_0)$ is a triangle in $\goth V$
  \par\indent 
  \strut\quad- then each isomorphism type of $\goth M$'s is determined by a 
  conjugacy class of $\sigma$ wrt. $\sim_{\Aut({\goth V})}$
  and we obtain a class of `new' configurations.
%%  class of $\sigma$  under the 
\end{enumerate}

To make this reasoning more complete let us quote after \cite[\brak]{maszko} one more result:
\begin{prop}[{\cite[A particular case of \brak]{maszko}}]\label{prop:iso1}
  Let $f$ map the points of $\perspace(p,\sigma_1,{\goth N}_1)$ onto the points of $\perspace(p,\sigma_2,{\goth N}_2)$, 
  $f(p) = p$, $\sigma_1,\sigma_2 \in S_{I_4}$,
  and 
  ${\goth N}_1, {\goth N}_2$ be two Veblen configurations defined
  on $\sub_2(I_4)$.
  The following conditions are equivalent.
  \begin{sentences}\itemsep-2pt
   \item\label{propiso1:war1}
     $f$ is an isomorphism 
     of $\perspace(p,\sigma_1,{\goth N}_1)$ onto $\perspace(p,\sigma_2,{\goth N}_2)$.
   \item\label{propiso1:war2}
     There is $\varphi \in S_{I_4}$ such that
\iffalse
     \begin{equation} \label{iso:war1}
       \overline{\varphi} \text{  is } \text{ an isomorphism of } 
       {\goth N}_1 \text{ onto } {\goth N}_2,
%%       \\ \label{iso:war2}
%%       \varphi \circ \sigma_1 & = & \sigma_2 \circ \varphi,
     \end{equation}
\fi
     one of the following holds
     \begin{eqnarray} 
     \label{iso:war1-1}
        \overline{\varphi}  & {\text is } & \text{ an isomorphism of } 
            {\goth N}_1 \text{ onto } {\goth N}_2,
     %%       \\ \label{iso0:war2}
     %%       \varphi \circ \sigma_1 & = & \sigma_2 \circ \varphi \text{ and }
 %%         \end{equation}
%%
     \\ 
     \label{propiso1:typ1}
       f(x_i) = x_{\varphi(i)},\; x = a,b,\; &&
       f(c_{\{i,j\}}) = c_{\{\varphi(i),\varphi(j)\}},
       \quad i,j\in I, i\neq j,
%%     \end{equation}
%%     \begin{equation} 
     \\ \label{iso:war2}
       \varphi \circ \sigma_1 & = & \sigma_2 \circ \varphi,
     \end{eqnarray}
     or
     \begin{eqnarray} 
     \label{iso:war1-2}
        \overline{\sigma_2^{-1}\varphi} & \text{  is } & \text{ an isomorphism of } 
            {\goth N}_1 \text{ onto } {\goth N}_2,
     %%       \\ \label{iso0:war2}
     %%       \varphi \circ \sigma_1 & = & \sigma_2 \circ \varphi \text{ and }
 %%         \end{equation}
%%
     \\ \nonumber
       f(a_i) = b_{\varphi(i)},\; f(b_i) = a_{\varphi(i)}, &&
        f(c_{\{i,j\}}) = c_{\{\sigma_2^{-1}\varphi(i),\sigma_2^{-1}\varphi(j)\}},
      \\  
     \label{propiso1:typ2}
%%        && \quad\quad\quad\quad\quad i,j\in I, i\neq j,
        && \strut\hspace{28ex}\strut i,j\in I, i\neq j,
%%     \end{equation}
%%     \begin{equation} 
     \\ \label{iso:war3}
       \varphi \circ \sigma_1 & = & \sigma_2^{-1} \circ \varphi.
     \end{eqnarray}
  \end{sentences}
\end{prop}

\par\medskip
In view of \ref{lem:nextgraf0} it will be usefull to know explicitly all the star-triangles in
corresponding labellings of the Veblen configuration.
Again, it is a school exercise to determine the following
\begin{lem}\label{triaINveb}
  The following are the star-triangles in corresponding Veblen configuration $\goth V$.
  \begin{description}\itemsep-2pt
  \item[{${\goth V} = \GrasSpace(I_4,2)$}:] $\starof(i)$, $i\in I_4$; 4 star-triangles.
  \item[{${\goth V} = \VeblSpace(2)$}:] $\starof(1)$, $\starof(2)$; 2 star-triangles.
  \item[{${\goth V} = {\cal V}_5$}:] $\starof(3)$; a unique star-triangle.
  \item[{${\goth V} = \GrasSpacex(I_4,2), {\cal V}_4, {\cal V}_6$}:] 
    no $i\in I_4$ such that $\starof(i)$ is a triangle in $\goth V$.
  \end{description}
\end{lem}

Finally, we are now in a position to formulate our main theorem classifying  perspectives
determined by permutations of indices.
\begin{thm}\label{glowne1}
  Let ${\goth M} = \perspace(p,\sigma,{\goth N})$, where $\sigma\in S_{I_4}$
  and $\goth N$ is a Veblen configuration defined on $\sub_2(I_4)$.
  Then $\goth M$ is isomorphic to exactly one of the following 
  ${\goth M}_0 = \perspace(p,\sigma_0,{\goth V})$.
  \begin{description}
  \item[{${\goth V} = \GrasSpace(I_4,2)$}:]
    \begin{typcription}
    \typitem{$C(\sigma) = (1,1,1,1)),\;\sigma_0 = \id$} 
      -- ${\goth M}_0 = \GrasSpace(6,2)$, it is a generalized Desargues Configuration 
      (or Cayley-Simson configuration).
    \typitem{$C(\sigma) = (2,2),\;\sigma_0 = (1,2)(3,4)$} 
      -- ${\goth M}_0 = {\goth R}_4$ is a quasi grassmannian of \cite{skewgras},
       cf. \cite[Example \brak]{maszko}.
    \typitem{$C(\sigma) = (1,3),\;\sigma_0 = (1)(2,3,4)$} 
      -- ${\goth M}_0$ has the type 2.8(ii) of \cite{STP3K5}, cf. \cite[Example \brak]{maszko}.
    \typitem{$C(\sigma) = (1,1,2),\;\sigma_0 = \id$}\label{typL4}
      -- ${\goth M}_0$ has the type 2.10(iii) of \cite{STP3K5},
      it is also isomorphic to the multiveblen configuration 
      $\xwlepp({I_4},{p},{L_{4}},{},{\GrasSpace(I_4,2)})$, 
      and to the skew perspective 
      $\perspace(p,\id,{\VeblSpace(2)})$ 
      (in this case the corresponding 
      $\VeblSpace(2)$ has lines $\topof(2)$ and $\topof(3)$), 
      comp. \cite[Example \brak]{maszko}.
    \typitem{$C(\sigma) = (4),\; \sigma_0 = (1,2,3,4)$} -- a new type. 
    \setcounter{ostitem}{\value{typek}} 
    \end{typcription}
  \item[{${\goth V} = \GrasSpacex(I_4,2)$}:]
    \begin{typcription}\setcounter{typek}{\theostitem}
    \typitem{$C(\sigma) = (1,1,1,1),\;\sigma_0 = \id$} -- a new type. 
    \typitem{$C(\sigma) = (2,2),\ \sigma = (1,2)(3,4)$} -- a new type. 
    \typitem{$C(\sigma) = (1,3),\ \sigma = (1)(2,3,4)$} -- a new type. 
    \typitem{$C(\sigma) = (1,1,2),\ \sigma = (1)(2)(3,4)$} -- a new type. 
    \typitem{$C(\sigma) = (4),\ \sigma = (1,2,3,4)$} -- a new type. 
    \setcounter{ostitem}{\value{typek}}
    \end{typcription}
  \item[{${\goth V} = \VeblSpace(2)$}:]
    \begin{typcription}\setcounter{typek}{\theostitem}
   \item $1,2\notin\Fix(\sigma)$
    \typitem{$C(\sigma) = (2,2), \;\sigma_0 = (1,2)(3,4)$} -- a new type.
    \typitem{$C(\sigma) = (2,2), \;\sigma_0 = (1,3)(2,4)$} -- a new type.
    \typitem{$C(\sigma) = (1,3), \;\sigma_0 = (3)(1,2,4)$} -- a new type.
    \typitem{$C(\sigma) = (4), \;\sigma_0 = (1,2,3,4)$} -- a new type.
    \typitem{$C(\sigma) = (4), \;\sigma_0 = (1,3,2,4)$} -- a new type.
   \item $1,2 \in \Fix(\sigma)$
    \typitem{$C(\sigma) = (1,1,2),\; \sigma_0 = (1)(2)(3,4)$}  
      -- ${\goth M}_0$ is the multiveblen configuration $\xwlepp({I_4},{p},{N_{4}},{},{\GrasSpace(I_4,2)})$,
      compare \cite[Exm. \brak]{maszko}, it has the type 2.10(ii) of \cite{STP3K5}.
      \par
      \strut -- Note that $\sigma_0 = id$ gives the structure already defined in \eqref{typL4} above.
   \item $1 \in \Fix(\sigma),\; 2 \notin \Fix(\sigma)$
    \typitem{$\sigma_0 = (1)(2,3,4)$}\label{class:1} 
      -- ${\goth M}_0$ has the type 2.8(xii) of \cite{STP3K5}.
    \typitem{$\sigma_0 = (1)(4)(2,3)$}\label{class:2} 
      -- ${\goth M}_0$ has the type 2.8(xi) of \cite{STP3K5}.
    \setcounter{ostitem}{\value{typek}}
    \end{typcription}
  \item[{${\goth V} = {\cal V}_4$}:]
    \begin{typcription}\setcounter{typek}{\theostitem}
    \typitem{$C(\sigma) = (1,1,1,1), \;\sigma_0 = \id$} -- a new type.
    \typitem{$C(\sigma) = (2,2), \;\sigma_0 = (1,2)(3,4)$} -- a new type.
    \typitem{$C(\sigma) = (2,2), \;\sigma_0 = (1,3)(2,4)$} -- a new type.
    \typitem{$C(\sigma) = (1,3), \;\sigma_0 = (1)(2,3,4)$} -- a new type.
    \typitem{$C(\sigma) = (1,3), \;\sigma_0 = (4)(1,2,3)$} -- a new type.
    \typitem{$C(\sigma) = (1,1,2), \;\sigma_0 = (1)(2)(3,4)$} -- a new type.
    \typitem{$C(\sigma) = (1,1,2), \;\sigma_0 = (1,2)(3)(4)$} -- a new type.
    \typitem{$C(\sigma) = (1,1,2), \;\sigma_0 = (1)(2,3)(4)$} -- a new type.
    \typitem{$C(\sigma) = (4), \;\sigma_0 = (1,2,3,4)$} -- a new type.
    \typitem{$C(\sigma) = (4), \;\sigma_0 = (1,3,2,4)$} -- a new type.
    \setcounter{ostitem}{\value{typek}}
    \end{typcription}
  \item[{${\goth V} = {\cal V}_5$}:]
    \begin{typcription}\setcounter{typek}{\theostitem}
    \item $3 \in \Fix(\sigma)$
    \typitem{$C(\sigma) = (1,1,1,1), \;\sigma_0 = \id$}\label{class:3} 
      -- ${\goth M}_0$ has the type 2.8(v) of \cite{STP3K5}.
    \typitem{$C(\sigma) = (1,3), \; \sigma = (3)(1,2,4)$}\label{class:4} 
      -- ${\goth M}_0$ has the type 2.8(i) of \cite{STP3K5}.
    \typitem{$C(\sigma) = (1,1,2), \;\sigma = (3)(1)(2,4)$}\label{class:5} 
      -- ${\goth M}_0$ has the type 2.8(xiv) of \cite{STP3K5}.
    \item $3 \notin \Fix(\sigma)$
    \typitem{$C(\sigma) = (2,2), \; \sigma = (1,2)(3,4)$} -- a new type.
    \typitem{$C(\sigma) = (1,3), \; \sigma = (1)(2,3,4)$} -- a new type.
    \typitem{$C(\sigma) = (1,1,2), \;\sigma = (1)(2)(3,4)$} -- a new type.
    \typitem{$C(\sigma) = (4), \;\sigma = (1,2,3,4)$} -- a new type.
    \setcounter{ostitem}{\value{typek}}
    \end{typcription}
  \item[{${\goth V} = {\cal V}_6$}:]
    \begin{typcription}\setcounter{typek}{\theostitem}
    \typitem{$C(\sigma) = (1,1,1,1), \;\sigma_0 = \id$} -- a new type.
    \typitem{$C(\sigma) = (2,2), \;\sigma_0 = (1,2)(3,4)$} -- a new type.
    \typitem{$C(\sigma) = (1,3), \;\sigma_0 = (1)(2,3,4)$} -- a new type.
    \typitem{$C(\sigma) = (1,3), \;\sigma_0 = (3)(1,2,4)$} -- a new type.
    \typitem{$C(\sigma) = (1,1,2), \;\sigma_0 = (1)(3)(2,4)$} -- a new type.
    \typitem{$C(\sigma) = (1,1,2), \;\sigma_0 = (1)(2)(3,4)$} -- a new type.
    \typitem{$C(\sigma) = (4), \;\sigma_0 = (1,2,3,4)$} -- a new type.
    \setcounter{ostitem}{\value{typek}}
    \end{typcription}
  \end{description}
\end{thm}
%
\iffalse
      \item[{$C(\beta) = (1,1,1,1)$}{\normalfont, then }] $\beta = \id$;
      \item[{$C(\beta) = (1,1,2)$}{\normalfont, then }] $\beta = (1)(2)(3,4),\;\; (1,2)(3)(4),\;\;(1)(2,3)(4)$;
      \item[{$C(\beta) = (1,3)$}{\normalfont, then }] $\beta = (1)(2,3,4),\;\; (4)(1,2,3)$;
      \item[{$C(\beta) = (2,2)$}{\normalfont, then }] $\beta = (1,2)(3,4),\;\; (1,4)(2,3)$;
      \item[{$C(\beta) = (4)$}{\normalfont, then }] $\beta = (1,2,3,4),\;\; (1,3,2,4)$.

      \item[$C(\beta) = (1,1,1,1)${\normalfont, then }] $\beta = \id$;
      \item[$C(\beta) = (1,1,2)${\normalfont, then }] $\beta = (1)(3)(2,4),\;\; (1)(2)(3,4)$;
      \item[$C(\beta) = (1,3)${\normalfont, then }] $\beta = (1)(2,3,4),\;\; (3)(1,2,4)$;
      \item[$C(\beta) = (2,2)${\normalfont, then }] $\beta = (1,2)(3,4)$;
      \item[$C(\beta) = (4)${\normalfont, then }] $\beta = (1,2,3,4)$.
\fi
%
\begin{proof}
  In view of \ref{lem:nextgraf0}, \ref{lem:vebauty}, \ref{lem:vebtypy}, and \ref{triaINveb} it only remains to
  justify that in the cases when $i\in\Fix(\sigma_0)$ and $\starof(i)$ is a triangle of $\goth V$
  the obtained ${\goth M}_0$ is of the appropriate type of \cite[Prop.'s 2.8, 2.10]{STP3K5}.
  Without coming into details we recall that a figure called 
  `the diagram of the line (which joins intersecting points of the three pairwise perspective 
  triangles in ${\goth M}_0$' and
  schemas of types of perspectives (labelled by symbols $\rho,\rho^{-1}, \sigma_x,\sigma_y$:
$\rho$ for $\raise3ex\hbox{\xymatrix{{}\ar@{-}[dr]&{}\ar@{-}[dr]&{}\ar@{-}[dll]\\{}&{}&{}}}$,
$\rho^{-1}$ for $\raise3ex\hbox{\xymatrix{{}\ar@{-}[drr]&{}\ar@{-}[dl]&{}\ar@{-}[dl]\\{}&{}&{}}}$,
and
$\sigma$ for one of $\raise3ex\hbox{\xymatrix{{}\ar@{-}[dr]&{}\ar@{-}[dl]&{}\ar@{-}[d]\\{}&{}&{}}}$)
  were used in \cite{STP3K5} to determine suitable types.
  Here, we indicate these labels, but we do not quote precise definitions, the more we do not cite the whole
  theory of \cite{STP3K5}.
  \begin{itemize}\itemsep-2pt\def\labelitemi{--}
  \item Ad \eqref{class:1}:
    Observe Figure \ref{fig:class:1}. In accordance with the notation of \cite{STP3K5},
    ${\goth M}_0$ is determined by the sequence $(\sigma_x,\sigma_x,\rho)$.
%%  \strut\quad\quad 
%%  The lines 
  \newline\strut\quad\quad
  $c_{2,3} \in \overline{c_{1,2},c_{1,4}},\; \overline{b_3,b_4},\; \overline{a_2,a_3}$,
%%  pass through $c_{2,3}$,
%%  the lines 
  $c_{3,4}\in \overline{c_{1,3},c_{1,4}},\; \overline{b_4,b_2},\; \overline{a_3,a_4}$,
%%  pass through $c_{3,4}$,
%%  and the lines 
  $c_{2,4} \in \overline{c_{1,2},c_{1,3}},\; \overline{b_3,b_2},\; \overline{a_2,a_4}$.
%%  pass through $c_{2,4}$.
  \newline\strut\quad\quad
  $b_1$ is the centre of $\Delta_1$ and $\Delta_2$,
  $a_1$ is the centre of $\Delta_2$ and $\Delta_3$,
  and
  $p$ is the centre of $\Delta_1$ and $\Delta_3$
  (lines in the diagram join points which correspond each to other under respective
  perspective).
  \item Ad \eqref{class:2}:
    Observe Figure \ref{fig:class:2}. In accordance with the notation of \cite{STP3K5},
    ${\goth M}_0$ is determined by the sequence $(\sigma_x,\sigma_x,\sigma_y)$.
%%  The lines 
  \newline\strut\quad\quad
  $c_{2,3} \in \overline{c_{1,2},c_{1,4}},\; \overline{b_3,b_2},\; \overline{a_2,a_3}$,
%%  pass through $c_{2,3}$,
%%  the lines 
  $c_{3,4}\in \overline{c_{1,3},c_{1,4}},\; \overline{b_4,b_2},\; \overline{a_3,a_4}$,
%%  pass through $c_{3,4}$,
%%  and the lines 
  $c_{2,4} \in \overline{c_{1,2},c_{1,3}},\; \overline{b_3,b_4},\; \overline{a_2,a_4}$.
%%  pass through $c_{2,4}$.
  \newline\strut\quad\quad
  $p$ is the centre of $\Delta_1$ and $\Delta_2$,
  $a_1$ is the centre of $\Delta_2$ and $\Delta_3$,
  and
  $b_1$ is the centre of $\Delta_1$ and $\Delta_3$
  \item Ad \eqref{class:3}:
    Observe Figure \ref{fig:class:3}. In accordance with the notation of \cite{STP3K5},
    ${\goth M}_0$ is determined by the sequence $(\rho,\rho^{-1},\id)$.
%%  The lines 
  \newline\strut\quad\quad
  $c_{1,2} \in \overline{c_{2,3},c_{3,4}},\; \overline{b_3,b_1},\; \overline{a_2,a_1}$,
%%  pass through $c_{2,3}$,
%%  the lines 
  $c_{2,4}\in \overline{c_{1,3},c_{3,4}},\; \overline{b_4,b_2},\; \overline{a_2,a_4}$,
%%  pass through $c_{3,4}$,
%%  and the lines 
  $c_{1,4} \in \overline{c_{2,3},c_{1,3}},\; \overline{b_1,b_4},\; \overline{a_1,a_4}$.
%%  pass through $c_{2,4}$.
  \newline\strut\quad\quad
  $b_3$ is the centre of $\Delta_1$ and $\Delta_2$,
  $a_3$ is the centre of $\Delta_2$ and $\Delta_3$,
  and
  $p$ is the centre of $\Delta_1$ and $\Delta_3$.
  \item Ad \eqref{class:4}:
    Observe Figure \ref{fig:class:4}. In accordance with the notation of \cite{STP3K5},
    ${\goth M}_0$ is determined by the sequence $(\rho,\rho,\rho)$.
%%  The lines 
  \newline\strut\quad\quad
  $c_{1,2} \in \overline{c_{2,3},c_{3,4}},\; \overline{b_4,b_1},\; \overline{a_2,a_1}$,
%%  pass through $c_{2,3}$,
%%  the lines 
  $c_{2,4}\in \overline{c_{1,3},c_{3,4}},\; \overline{b_1,b_2},\; \overline{a_2,a_4}$,
%%  pass through $c_{3,4}$,
%%  and the lines 
  $c_{1,4} \in \overline{c_{2,3},c_{1,3}},\; \overline{b_2,b_4},\; \overline{a_1,a_4}$.
%%  pass through $c_{2,4}$.
  \newline\strut\quad\quad
  $b_3$ is the centre of $\Delta_1$ and $\Delta_2$,
  $p$ is the centre of $\Delta_2$ and $\Delta_3$,
  and
  $a_3$ is the centre of $\Delta_1$ and $\Delta_3$.
  \item Ad \eqref{class:5}:
    Observe Figure \ref{fig:class:5}. In accordance with the notation of \cite{STP3K5},
    ${\goth M}_0$ is determined by the sequence $(\rho,\rho^{-1},\sigma_x)$.
%%  The lines 
  \newline\strut\quad\quad
  $c_{1,2} \in \overline{c_{2,3},c_{3,4}},\; \overline{b_4,b_1},\; \overline{a_2,a_1}$,
%%  pass through $c_{2,3}$,
%%  the lines 
  $c_{2,4}\in \overline{c_{1,3},c_{3,4}},\; \overline{b_4,b_2},\; \overline{a_2,a_4}$,
%%  pass through $c_{3,4}$,
%%  and the lines 
  $c_{1,4} \in \overline{c_{2,3},c_{1,3}},\; \overline{b_1,b_2},\; \overline{a_1,a_4}$.
%%  pass through $c_{2,4}$.
  \newline\strut\quad\quad
  $b_3$ is the centre of $\Delta_1$ and $\Delta_2$,
  $a_3$ is the centre of $\Delta_2$ and $\Delta_3$,
  and
  $p$ is the centre of $\Delta_1$ and $\Delta_3$.
  \end{itemize}
  This completes our proof.
\end{proof}

\begin{figure}[t]  %% zlaczona!!!
\begin{center}  
  \begin{minipage}[m]{0.45\textwidth}  %%% class:1
   \begin{center} 
   \xymatrix{%
    {\Delta_1:}
    &
    {a_{4}}\ar@{-}[dr]\ar@{-}[ddr]
    &
    {a_{3}}\ar@{-}[dl]\ar@{-}[ddr]
    &
    {a_{2}}\ar@{-}[d]\ar@(dr,ur)@{-}[ddll]
    \\
    {\Delta_2:}
    &
    {c_{1,3}}\ar@{-}[dr]
    &
    {c_{1,4}}\ar@{-}[dl]
    &
    {c_{1,2}}\ar@{-}[d]
    \\
    {\Delta_3:}
    &
    {b_{2}}
    &
    {b_{4}}
    &
    {b_{3}}
    }%% ENDofxy
    \end{center}
    \refstepcounter{figure}\label{fig:class:1}
    \begin{center}
    Figure \ref{fig:class:1}.
    The diagram of the line $\{ c_{2,3}, c_{2,4}, c_{3,4} \}$
  in ${\goth M}_0$.\myend
    \end{center}
  \end{minipage}  %% class:1 - end
  \quad
%%%
  \begin{minipage}[m]{0.45\textwidth} %% class:2
  \begin{center}
  \xymatrix{%
    {\Delta_1:}
    &
    {b_{3}}\ar@{-}[d]\ar@{-}[ddr]
    &
    {b_{2}}\ar@{-}[dr]\ar@{-}[ddl]
    &
    {b_{4}}\ar@{-}[dl]\ar@(dr,ur)@{-}[dd]
    \\
    {\Delta_2:}
    &
    {c_{1,2}}\ar@{-}[d]
    &
    {c_{1,4}}\ar@{-}[dr]
    &
    {c_{1,3}}\ar@{-}[dl]
    \\
    {\Delta_3:}
    &
    {a_{2}}
    &
    {a_{3}}
    &
    {a_{4}}
    }%% ENDofxy
  \end{center}
    \refstepcounter{figure}\label{fig:class:2}
    \begin{center}
    Figure \ref{fig:class:2}. The diagram of the line $\{ c_{2,3}, c_{2,4}, c_{3,4} \}$
  in ${\goth M}_0$. \myend
    \end{center}
  \end{minipage}  %%%%%%%%%% class:2 -end
 \end{center}
%%  \caption{The diagram of the line $\{ c_{2,3}, c_{2,4}, c_{3,4} \}$
%%  in ${\goth M}_0$. 
%% \myend} %% ENDofCAPTION
%% \label{fig:class:1-2}
\end{figure}

\iffalse
\begin{figure} %% stara pojedyncza, class:2
\begin{center}
  \begin{minipage}[m]{0.45\textwidth}
  \xymatrix{%
    {\Delta_1:}
    &
    {b_{3}}\ar@{-}[d]\ar@{-}[ddr]
    &
    {b_{2}}\ar@{-}[dr]\ar@{-}[ddl]
    &
    {b_{4}}\ar@{-}[dl]\ar@(dr,ur)@{-}[dd]
    \\
    {\Delta_2:}
    &
    {c_{1,2}}\ar@{-}[d]
    &
    {c_{1,4}}\ar@{-}[dr]
    &
    {c_{1,3}}\ar@{-}[dl]
    \\
    {\Delta_3:}
    &
    {a_{2}}
    &
    {a_{3}}
    &
    {a_{4}}
    }%% ENDofxy
  \end{minipage}
\end{center}
  \caption{The diagram of the line $\{ c_{2,3}, c_{2,4}, c_{3,4} \}$
  in ${\goth M}_0$. 
\myend} %% ENDofCAPTION
\label{fig:class:2-0}
\end{figure}
\fi

\begin{figure}[t] %%%% zlaczona
\begin{center}
  \begin{minipage}[m]{0.45\textwidth} %%%%  class:3
  \begin{center}
    \xymatrix{%
    {\Delta_1:}
    &
    {a_{4}}\ar@{-}[dr]\ar@(dl,ul)@{-}[dd]
    &
    {a_{2}}\ar@{-}[dr]\ar@(dl,ul)@{-}[dd]
    &
    {a_{1}}\ar@{-}[dll]\ar@(dr,ur)@{-}[dd]
    \\
    {\Delta_2:}
    &
    {c_{1,3}}\ar@{-}[drr]
    &
    {c_{3,4}}\ar@{-}[dl]
    &
    {c_{2,3}}\ar@{-}[dl]
    \\
    {\Delta_3:}
    &
    {b_{4}}
    &
    {b_{2}}
    &
    {b_{1}}
    }%% ENDofxy
  \end{center}
  \refstepcounter{figure}\label{fig:class:3}
  \begin{center}
  Figure \ref{fig:class:3}. The diagram of the line $\{ c_{1,2}, c_{1,4}, c_{2,4} \}$
  in ${\goth M}_0$. \myend
  \end{center}
  \end{minipage} %%% class:3 -end
  \quad
%%%%  
  \begin{minipage}[m]{0.45\textwidth} %%% class:4
  \begin{center}
    \xymatrix{%
    {\Delta_1:}
    &
    {c_{2,3}}\ar@{-}[dr]\ar@{-}[ddr]
    &
    {c_{3,4}}\ar@{-}[dr]\ar@{-}[ddr]
    &
    {c_{1,3}}\ar@{-}[dll]\ar@(dr,ur)@{-}[ddll]
    \\
    {\Delta_2:}
    &
    {a_{1}}\ar@{-}[dr]
    &
    {a_{2}}\ar@{-}[dr]
    &
    {a_{4}}\ar@{-}[dll]
    \\
    {\Delta_3:}
    &
    {b_{4}}
    &
    {b_{1}}
    &
    {b_{2}}
    }%% ENDofxy
  \end{center}
  \refstepcounter{figure}\label{fig:class:4}
  \begin{center}
  Figure \ref{fig:class:4}. The diagram of the line $\{ c_{1,2}, c_{2,3}, c_{2,4} \}$
  in ${\goth M}_0$. \myend
  \end{center}
  \end{minipage} %%% class:4 -end
\end{center}
%% \label{fig:class:3-4}
%%   \caption{The diagram of the line $\{ c_{1,2}, c_{1,4}, c_{2,4} \}$
%%   in ${\goth M}_0$. 
%% \myend} %% ENDofCAPTION
%% %% \label{fig:class:3-4}
\end{figure}

\iffalse
\begin{figure}
\begin{center}
  \begin{minipage}[m]{0.6\textwidth}
    \xymatrix{%
    {\Delta_1:}
    &
    {c_{2,3}}\ar@{-}[dr]\ar@{-}[ddr]
    &
    {c_{3,4}}\ar@{-}[dr]\ar@{-}[ddr]
    &
    {c_{1,3}}\ar@{-}[dll]\ar@(dr,ur)@{-}[ddll]
    \\
    {\Delta_2:}
    &
    {a_{1}}\ar@{-}[dr]
    &
    {a_{2}}\ar@{-}[dr]
    &
    {a_{4}}\ar@{-}[dll]
    \\
    {\Delta_3:}
    &
    {b_{4}}
    &
    {b_{1}}
    &
    {b_{2}}
    }%% ENDofxy
  \end{minipage}
\end{center}
  \caption{The diagram of the line $\{ c_{1,2}, c_{1,4}, c_{2,4} \}$
  in ${\goth M}_0$. 
\myend} %% ENDofCAPTION
\label{fig:class:4}
\end{figure}
\fi

\begin{figure}[b] %%% ostatnia
\begin{center}
  \begin{minipage}[m]{0.45\textwidth} %% class:5
    \xymatrix{%
    {\Delta_1:}
    &
    {a_{4}}\ar@{-}[dr]\ar@{-}[ddr]
    &
    {a_{2}}\ar@{-}[dr]\ar@{-}[ddl]
    &
    {a_{1}}\ar@{-}[dll]\ar@(dr,ur)@{-}[dd]
    \\
    {\Delta_2:}
    &
    {c_{1,3}}\ar@{-}[drr]
    &
    {c_{3,4}}\ar@{-}[dl]
    &
    {c_{2,3}}\ar@{-}[dl]
    \\
    {\Delta_3:}
    &
    {b_{2}}
    &
    {b_{4}}
    &
    {b_{1}}
    }%% ENDofxy
  \end{minipage}
\end{center}
  \caption{The diagram of the line $\{ c_{1,2}, c_{1,4}, c_{2,4} \}$
  in ${\goth M}_0$. 
\myend} %% ENDofCAPTION
\label{fig:class:5}
\end{figure}

Combining \ref{prop:iso1}, \ref{lem:vebauty}, and \ref{lem:vebtypy} one can write down
explicitly the list of automorphism groups of the structures defined in \ref{glowne1};
we pass over this task, as no new essential information can be obtained on this way.

%%%%%%%%%%%%%%%%%%%%%%%%%%%%%%%%%%%%%%%%%%%%%%%%%%%%%%%%%%
%%%%%%%%%%%%%%%%%%%%%%%%%%%%%%%%%%%%%%%%%%%%%%%%%%%%%%%%%%
%%%%%%%%%%%%%%%%%%%%%%%%%%%%%%%%%%%%%%%%%%%%%%%%%%%%%%%%%% sec:complementations
%%%%%%%%%%%%%%%%%%%%%%%%%%%%%%%%%%%%%%%%%%%%%%%%%%%%%%%%%%
%%%%%%%%%%%%%%%%%%%%%%%%%%%%%%%%%%%%%%%%%%%%%%%%%%%%%%%%%%

\section{Perspectivities associated with boolean complementing}\label{sec:bool-pers}

Here, we shall pay attention to the structures $\perspace(p,\sigma,{\goth N})$,
where $\goth N$ is a $\konftyp(6,2,4,3)$-configuration 
(i.e. it is the Veblen configuration, suitably labelled, defined on $\sub_2(I_4)$)
and
$\sigma\in S_{\sub_2(I_4)}$ is defined as a composition of the boolean complementing 
in $\sub_2(I_4)$ and a map determined by a permutation of $I_4$ (cf.  definition of a skew perspective). %% \ref{lem:meetinvar}).
Consequently, we obtain another class of $\konftyp(15,4,20,3)$-configurations. 
%% (comp. Subsec. \ref{subsec:konter}).
\par

%% Let us write $\varkappa$ for the boolean complementing considered on $\sub_2(I_4)$.
%
So, let
$\sigma = \varkappa \overline{\varphi}$ where $\varphi \in S_{I_4}$.
It is known that
\begin{equation*}
  \varkappa \overline{\varphi} = \overline{\varphi} \varkappa \; \text{ for every }
  \varphi \in S_{I_4}.
\end{equation*}
\def\kapers(#1,#2){{\goth K}_{{#1},{#2}}}
Let us write, for short
\begin{ctext}
  $\perspace(p,{\overline{\varphi}\varkappa},{\goth N}) =\colon \kapers(\varphi,{\goth N})$;
\end{ctext}
let $\goth N$ be any Veblen configuration defined on $\sub_2(I_4)$.
The following formula, valid in $\kapers(\varphi,{\goth N})$, will be frequently used
without explicit quotation:
\begin{ctext}
  $b_i \oplus b_j = c_{\varkappa\overline{\varphi^{-1}}(\{ i,j \})}$.
\end{ctext}

\medskip
As we know, $A^\ast$ and $B^\ast$ are two $K_5$ graphs freely contained in 
$\kapers(\varphi,{\goth N})$.
\begin{lem}\label{prop:bool:bezgraf}
  $\kapers(\varphi,{\goth N})$ does not freely contain any other $K_5$-graph.
\end{lem}
\begin{proof}
  Suppose that $G\neq K_{A^\ast},K_{B^\ast}$ is a complete $K_5$ graph freely contained in
  $\kapers(\varphi,{\goth N})$.
  Arguing as in the proof of \cite[Lem. \brak]{maszko} we find $i_0\in I_4$ such that
%%  $p,a_{i_0},b_{i_0}$ are on a line of $\goth M$ and
  $a_{i_0}\in A,G$, $b_{i_0}\in B,G$. 
  Then 
  $G\setminus\{ a_{i_0},b_{i_0} \} \subset \starof(i_0) \subset C$.
  So, $c_{i_0,x},b_{i_0}$ must colline for every $x\in I_4\setminus\{i_0\} =: I'$;
  this means: for every $x\in I'$ there is $x'\in I$ such that 
  $c_{i_0,x} = b_{i_0} \oplus b_{x'} = c_{\varkappa\overline{\varphi^{-1}}(\{ i_0,x' \})}$.
  Write $I_4 = \{ i_0,j,k,l \}$. Then we obtain
  \begin{ctext}
    $\{ \varphi^{-1}(i_0),\varphi^{-1}(j') \} = \{ k,l \}$,
    $\{ \varphi^{-1}(i_0),\varphi^{-1}(l') \} = \{ k,j \}$, and
    $\{ \varphi^{-1}(i_0),\varphi^{-1}(k') \} = \{ j,l \}$.
  \end{ctext}
  Consequently, there is no room for $\varphi^{-1}(i_0)$.
\end{proof}
As an immediate consequence of \ref{prop:bool:bezgraf} we conclude with
\begin{cor}\label{cor:bool:fix}
  $f(p) = p$ for every $f\in \Aut({\kapers(\varphi,{\goth N})})$.
\end{cor}
\begin{lem}\label{lem:perNOkap}
  There is no $\sigma\in S_{I_4}$ such that 
  $\kapers(\varphi,{\goth N}) \cong \perspace(p,\overline{\sigma},{\goth N}')$
  for a Veblen configuration ${\goth N}'$.
\end{lem}
\begin{proof}
  Suppose that an isomorphism $f$ exists which maps 
  $\perspace(p,\overline{\sigma},{\goth N}')$ onto $\kapers(\varphi,{\goth N})$.
  Then $f(p) = p$ and either $f$ maps $A$ onto $A$, $B$ onto $B$, or $f$ interchanges
  $A$ and $B$. There exists $\alpha\in S_{I_4}$ such that $f(a_i) = a_{\alpha(i)}$
  for each $i \in I_4$ or $f(a_i) = b_{\alpha(i)}$ for each $i$.
  In the first case we obtain 
  \begin{math}
   c_{\overline{\alpha}\overline{\sigma^{-1}}(\{i,j\} )} = 
   f(c_{\overline{\sigma^{-1}}\{i,j\} }) =
   f(b_i \oplus b_j) = f(b_i) \oplus f(b_j) = b_{\alpha(i)} \oplus b_{\alpha(j)} =
   c_{ \varkappa\overline{\varphi^{-1}}\overline{\alpha}( \{ i,j \} ) }.
  \end{math}
  This gives, inconsistently, 
  $\varkappa = \overline{\alpha \sigma^{-1} \alpha^{-1} \varphi }$.
  \par\noindent
  In the second case, with analogous computation we obtain
  $f(c_{ \{ i,j \} }) = c_{\varkappa\overline{\varphi^{-1}\alpha}( \{ i,j \} ) }$
  and then we arrive to 
  $\varkappa = \overline{\varphi^{-1} \alpha \sigma^{-1} \alpha^{-1} }$.
\end{proof}
The following is seen:
\begin{lem}\label{lem:kap-wymian}
  Let $\goth N$ be a Veblen configuration defined on $\sub_2(I_4)$.
  Then the map
  \begin{equation}
    p\mapsto p,\quad a_i \mapsto b_i \mapsto a_i,\, i\in I_4, \quad
    c_u \mapsto c_{\varkappa(u)},\, u\in \sub_2(I_4)
  \end{equation}
%%  Then a map which interchanges $A$ and $B$ 
  is an isomorphism of 
  $\perspace(p,\id,{\goth N})$ and $\perspace(p,\id,\varkappa({\goth N}))$.
\end{lem}

Now, we are in a position to prove an analogue of \ref{prop:iso1}.
\begin{prop}\label{prop:iso2}
  Let $\varphi_1,\varphi_2 \in S_{I_4}$, ${\goth N}_1,{\goth N}_2$ be two
  Veblen configurations defined on $\sub_2(I_4)$.
  The following conditions are equivalent.
  \begin{sentences}\itemsep-2pt
  \item\label{propiso2:war0} 
    $f$ is an isomorphism of 
    $\kapers({\varphi_1},{\goth N}_1)$ onto $\kapers({\varphi_2},{\goth N}_2)$.
  \item
    There is $\alpha\in S_{I_4}$ such that 
    one of the following holds
    %
%%    \begin{equation}\label{propiso2:war1}
%%      \overline{\alpha} \text{ is an isomorphism of } {\goth N}_1 \text{ onto } {\goth N}_2
%%    \end{equation}
    %
    %
    \begin{enumerate}[\rm(a)]\itemsep-2pt
    \item\label{propiso2:typ1} 
      \eqref{propiso1:typ1} and 
      \begin{eqnarray} \label{propiso2:war1}
        \overline{\alpha} & \text{ is }& \text{ an isomorphism of } {\goth N}_1 \text{ onto } {\goth N}_2,
	\\ \label{propiso2:war2}
        \varphi_2 & =  & {\varphi_1}^{\alpha}, 
      \end{eqnarray}
      or
    \item \label{propiso2:typ2}
     \begin{eqnarray} \nonumber
       f(a_i) \quad = \quad b_{\alpha(i)},\quad f(b_i) & = & a_{\alpha(i)}, 
       \\ \label{propiso2:war3}
        f(c_{\{i,j\}}) & = & c_{\varkappa\overline{\varphi_2^{-1}\alpha}(\{i,j\})},
       \; \{i,j\}\in \sub_2(I_4),
%%     \end{equation}
%%     \begin{equation} 
     \\ \label{propiso2:war4}
      \varkappa\overline{\varphi_2^{-1}\alpha} &\text{ is }& \text{ an isomorphism of } {\goth N}_1 \text{ onto } {\goth N}_2
     \\ \label{propiso2:war5}
        \varphi_2^{-1} & = & {\varphi_1}^{\alpha}
    \end{eqnarray}
%%      \eqref{propiso1:typ2} and $\varphi_2^{-1} = {\varphi_1}^{-1}$
    \end{enumerate}
    (cf. \eqref{iso:war2} and \eqref{iso:war3} resp.). Note that \eqref{propiso2:war5} implies 
    $\varkappa \overline{\varphi_2^{-1}\alpha} = \varkappa \overline{\alpha\varphi_1}$.
  \end{sentences}
\end{prop}
\begin{proof}
  The proof is analogous to the proof of  \cite[Prop. \brak]{maszko} (cf. \ref{prop:iso1}).
  First, given an isomorphism $f$ as in \eqref{propiso2:war0} 
  from \ref{prop:bool:bezgraf} we have $f(p) = p$ and,
  consequently, either $f(A) = A$, $f(B) = B$, or $f(A) = B$, $f(B) = A$.
  In both cases there is an $\alpha\in S_{I_4}$ such that
  $f(a_i) = a_{\alpha(i)}$ for $i\in I_4$ in the first case, 
  and $f(a_i) = b_{\alpha(i)}$ in the second one.
  Then we obtain 
  $f(c_{ \{i,j \} }) = c_{\overline{\alpha}(\{ i,j\}) }$
  in the first case, and
  $f(c_{ \{i,j \} }) = c_{\varkappa \overline{\varphi_2^{-1}\alpha}(\{ i,j\}) }$
  in the second one, for all $\{ i,j \}\in\sub_2(I_4)$.
  This justifies \eqref{propiso1:typ1} and \eqref{propiso2:war3} in corresponding cases.
  Besides, since $f\restriction{C}$ is an isomorphism of ${\goth N}_1$ onto ${\goth N}_2$,
  we arrive to \eqref{propiso2:war1} and \eqref{propiso2:war4}, respectively.
  Finally, computing $f(b_i\oplus b_j)$ we obtain
  $\varkappa \overline{\alpha \varphi^{-1}} = \varkappa \overline{\varphi_2^{-1}}$
  in the first case
  and
  $\overline{\varphi_2 \alpha \varphi_1} = \overline{\alpha}$ in the second case.
  So, we have proved \eqref{propiso2:war2} and \eqref{propiso2:war5} 
  in the corresponding cases.
  \par
  A routine, though quite tidy computation justifies that a map $f$, when characterized 
  by \eqref{propiso2:typ1} or by \eqref{propiso2:typ2}, is an isomorphism
  as in \eqref{propiso2:war0} required.
\end{proof}
Let us write down a direct consequence of \ref{prop:iso2}
\begin{cor}\label{cor:2standard}
  Let $\alpha\in S_{I_4}$.
  Then
  $\perspace(p,\overline{\varphi}\varkappa,{\goth N}) \cong
   \perspace(p,\overline{\varphi^\alpha}\varkappa,{\overline{\alpha}({\goth N})})$.
\end{cor}

\iffalse
As a consequence, we obtain exactly three types of respective skew perspectives:
they are presented in the figures \ref{fig:skos1}-\ref{fig:skos3}.
Names of the labellings are quoted after \cite{klik:VC}
\fi
\ifhozna\relax\else

\begin{figure}
\begin{center}
  \includegraphics[scale=0.5]{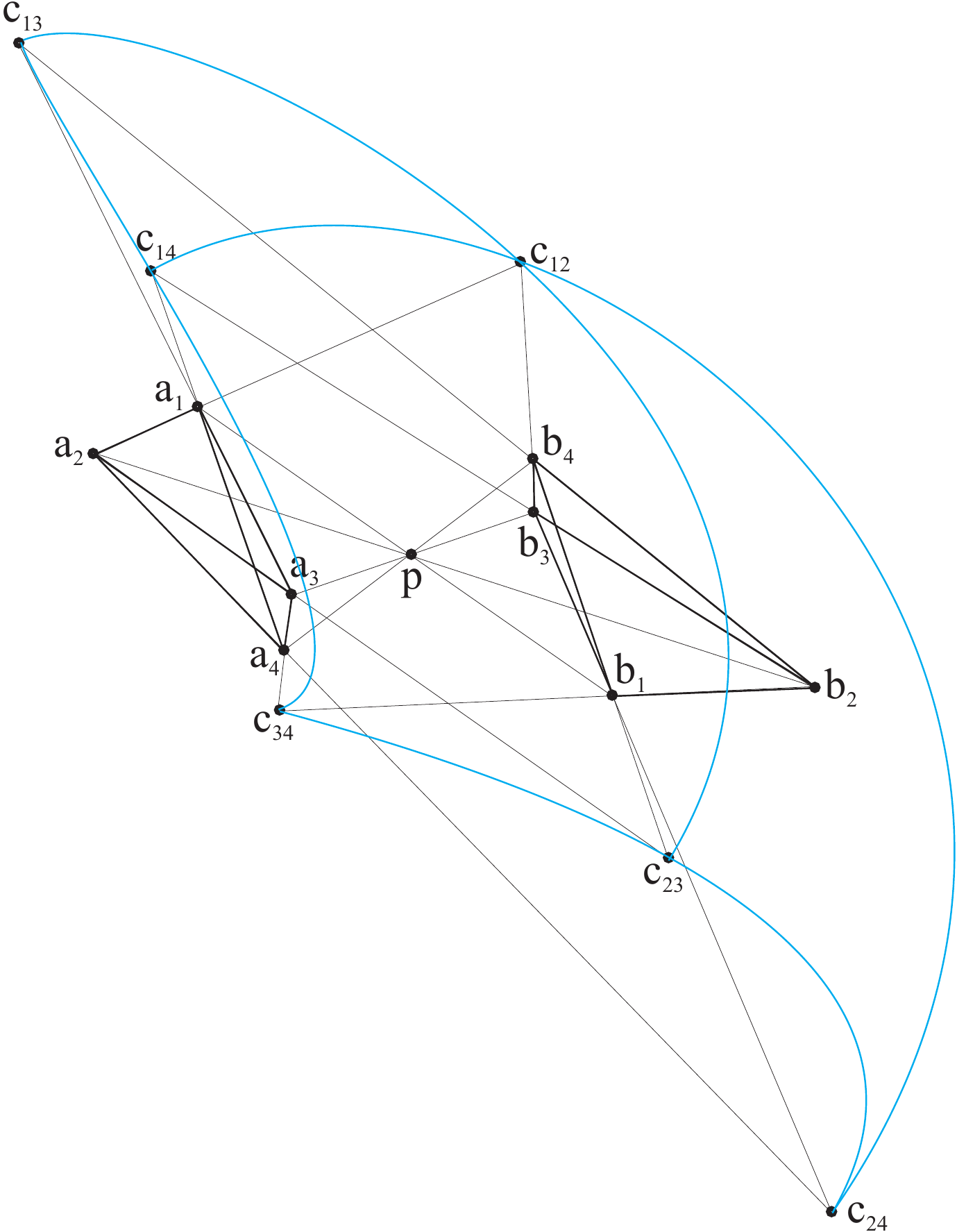}
\end{center}
\caption{The configuration $\perspace(4,\id,{\GrasSpace(4,2)})$}
\label{fig:skos1}
\end{figure}

\begin{figure}
\begin{center}
  \includegraphics[scale=0.5]{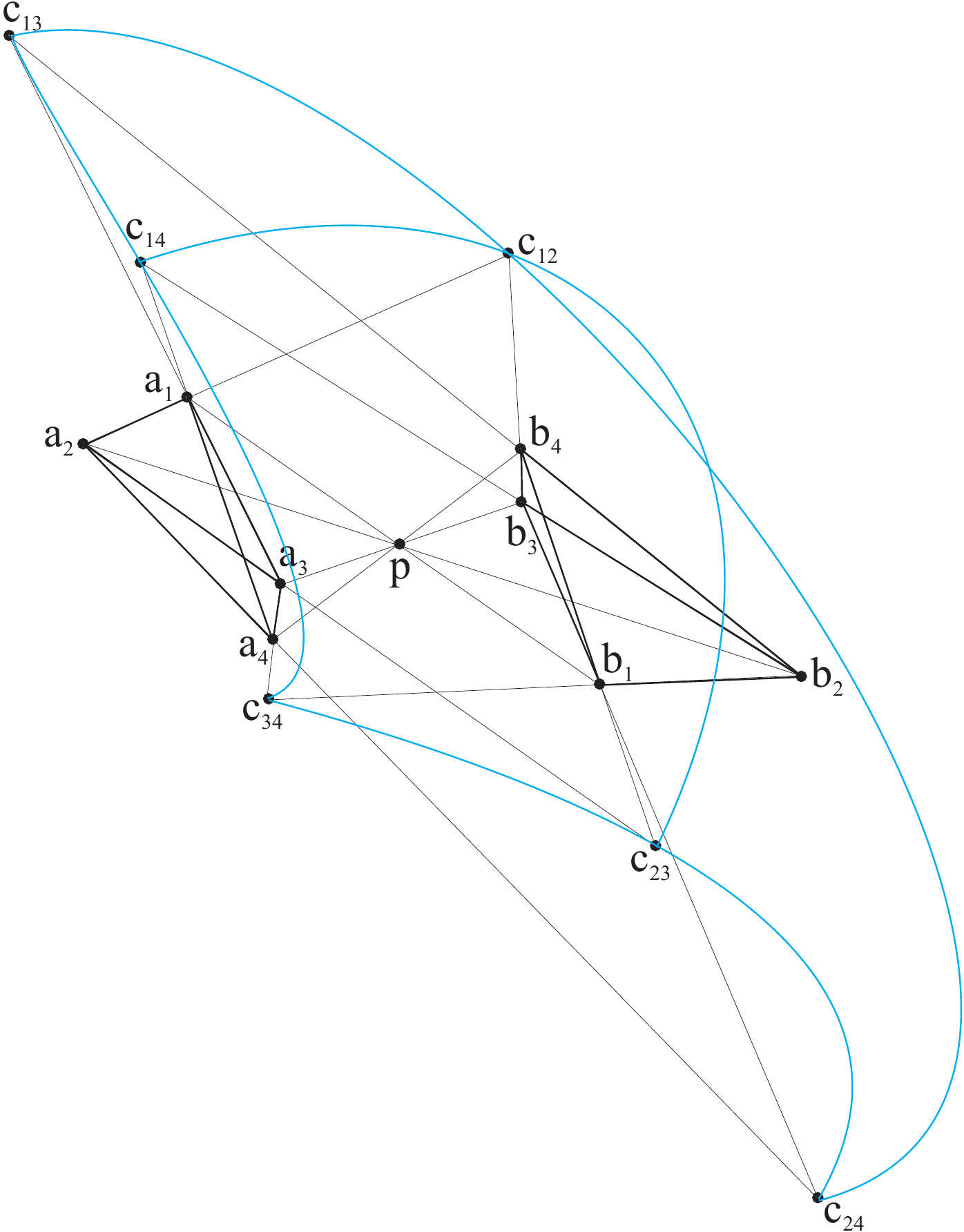}
\end{center}
\caption{The configuration $\perspace(4,\id,{\VeblSpace(2)})$}
\label{fig:skos2}
\end{figure}

\begin{figure}
\begin{center}
  \includegraphics[scale=0.5]{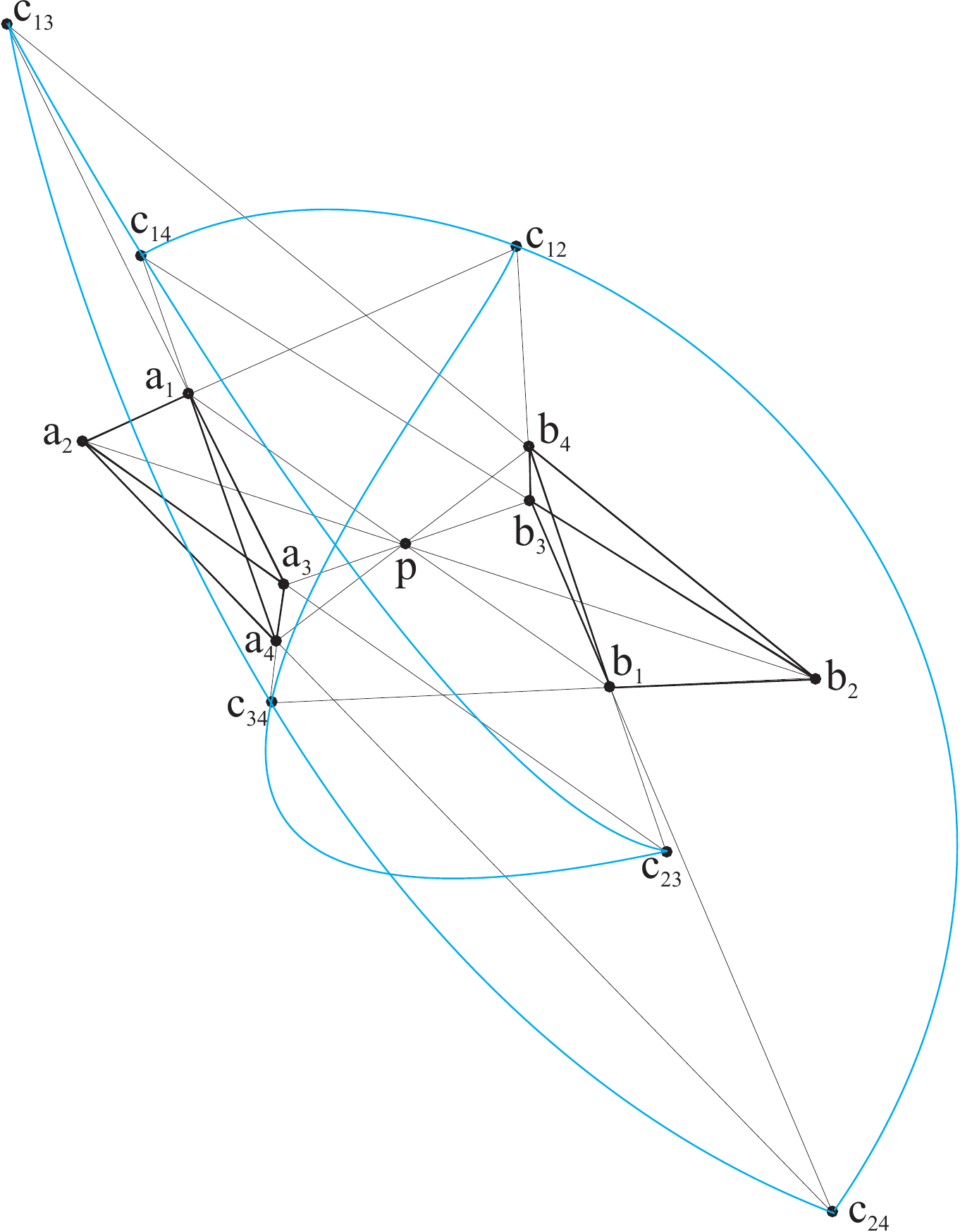}
\end{center}
\caption{The configuration $\perspace(4,\id,{\cal V}_5)$}
\label{fig:skos3}
\end{figure}

\fi

In view of \ref{fct:vebetyk}, \ref{cor:2standard}, and \ref{lem:kap-wymian} we obtain the following,
rather rough, yet, classification.
\begin{prop}\label{bool:klasif1}
  For every $\varphi\in S_{I_4}$ and every Veblen configuration $\goth N$ defined on $\sub_2(I_4)$
  there is $\beta\in S_{I_4}$ such that 
  $\kapers(\varphi,{\goth N}) \cong \kapers(\beta,{\goth V})$,
  where $\goth V$ is a structure in the list \eqref{veblist1}.
\iffalse  
  one of the following holds
  %
  \begin{enumerate}[{\ref{bool:klasif1}.}(a)]\itemsep-2pt
  \item\label{veb:typ1}
    ${\goth V} = \GrasSpace(I_4,2)$, that can be recognized in Figure \ref{fig:skos1},
  \item\label{veb:typ2}
    ${\goth V} = \VeblSpace(2)$, defined in Figure \ref{fig:skos2}, and
  \item\label{veb:typ3}
    ${\goth V} = {\cal V}_5$, defined in Figure \ref{fig:skos3}.
  \end{enumerate}
\fi
\end{prop}
\def\refv#1{\ref{bool:klasif1}.\eqref{#1}}
Consequently, we only need to classify all the structures $\kapers(\beta,{\goth V})$.
Recall, that there is no $\alpha\in S_{I_4}$ such that
$\varkappa\overline{\alpha}\in\Aut({\goth V})$, where $\goth V$ is among structures
listed in \eqref{veblist1}.
Consequently, from \ref{prop:iso2} we conclude with the following
\begin{lem}
  Let $\beta_1,\beta_2\in S_{I_4}$, 
  let $\goth V$ be among structures defined in \eqref{veblist1}.
  Then $\kapers(\beta_1,{\goth V}) \cong \kapers(\beta_2,{\goth V})$ iff
  $\beta_1,\beta_2$ are conjugate under an $\alpha$ such that 
  $\overline{\alpha}\in\Aut({\goth V})$.
\end{lem}
Thus, substituting \ref{lem:vebauty} and \ref{lem:vebtypy} to the definition of 
$\perspace(p,\overline{\varphi}\varkappa,{\goth N})$
finally, we obtain our final classification.
\begin{thm}\label{glowne2}
  Let ${\goth M} = \perspace(p,\overline{\varphi}\varkappa,{\goth N})$,
  where $\goth N$ is a Veblen configuration defined on $\sub_2(I_4)$.
  Then $\goth M$ is isomorphic to (exactly) one of the following
  %
%  \begin{enumerate}[\rm(i)]\itemsep-2pt
%  \item
\setcounter{enumi}{0}
\refstepcounter{enumi}  \label{klas:typ1}
\begin{ctext}(\theenumi)\quad    $\perspace(4,\varkappa,{\GrasSpace(I_4,2)})$,\space
    $\perspace(4,\overline{(1,2),(3,4)}\varkappa,{\GrasSpace(I_4,2)})$,\space
    $\perspace(4,\overline{(1),(2),(3,4)}\varkappa,{\GrasSpace(I_4,2)})$,\space
    $\perspace(4,\overline{(1),(2,3,4)}\varkappa,{\GrasSpace(I_4,2)})$,\space
    $\perspace(4,\overline{(1,2,3,4)}\varkappa,{\GrasSpace(I_4,2)})$. %% 5 items
\end{ctext}
\vskip-0.5ex
%  \item
\refstepcounter{enumi}  \label{klas:typ2}
\begin{ctext}(\theenumi)\quad    $\perspace(4,\varkappa,{\VeblSpace(2)})$,\space
    $\perspace(4,\overline{(1)(2)(3,4)}\varkappa,{\VeblSpace(2)})$,\space
    $\perspace(4,\overline{(1,2)(3)(4)}\varkappa,{\VeblSpace(2)})$,\space
    $\perspace(4,\overline{(1)(2,3,4)}\varkappa,{\VeblSpace(2)})$,\space
    $\perspace(4,\overline{(4)(1,2,3)}\varkappa,{\VeblSpace(2)})$,\space
    $\perspace(4,\overline{(1,2)(3,4)}\varkappa,{\VeblSpace(2)})$,\space
    $\perspace(4,\overline{(1,4)(2,3)}\varkappa,{\VeblSpace(2)})$,\space
    $\perspace(4,\overline{(1,2,3,4)}\varkappa,{\VeblSpace(2)})$.  %% 8 items
%%    $\perspace(4,\varkappa,{\VeblSpace(2)})$,
%%    $\perspace(4,\varkappa,{\VeblSpace(2)})$,
%%    $\perspace(4,\varkappa,{\VeblSpace(2)})$,
%%    $\perspace(4,\varkappa,{\VeblSpace(2)})$,
%%    $\perspace(4,\varkappa,{\VeblSpace(2)})$,
%%    $\perspace(4,\varkappa,{\VeblSpace(2)})$,
%%
\end{ctext}
\vskip-0.5ex
%  \item
\refstepcounter{enumi} \label{klas:typ3}
\begin{ctext}(\theenumi)\quad    $\perspace(4,{}\varkappa,{{\cal V}_5})$,\space
    $\perspace(4,\overline{(1)(3)(2,4)}\varkappa,{{\cal V}_5})$,\space
    $\perspace(4,\overline{(1)(2)(3,4)}\varkappa,{{\cal V}_5})$,\space
    $\perspace(4,\overline{(1)(2,3,4)}\varkappa,{{\cal V}_5})$,\space
    $\perspace(4,\overline{(3)(1,2,4)}\varkappa,{{\cal V}_5})$,\space
    $\perspace(4,\overline{(1,2,3,4)}\varkappa,{{\cal V}_5})$,\space
    $\perspace(4,\overline{(1,2)(3,4)}\varkappa,{{\cal V}_5})$. %% 7 items
%%    $\perspace(4,{}\varkappa,{{\cal V}_5})$,
%%    $\perspace(4,{}\varkappa,{{\cal V}_5})$,
%%    $\perspace(4,{}\varkappa,{{\cal V}_5})$,
%%    $\perspace(4,{}\varkappa,{{\cal V}_5})$,
%%    $\perspace(4,{}\varkappa,{{\cal V}_5})$,
%%    $\perspace(4,{}\varkappa,{{\cal V}_5})$,
%%    $\perspace(4,{}\varkappa,{{\cal V}_5})$,
%%    $\perspace(4,{}\varkappa,{{\cal V}_5})$,
%%    $\perspace(4,{}\varkappa,{{\cal V}_5})$,
%%    $\perspace(4,{}\varkappa,{{\cal V}_5})$,
%%    $\perspace(4,{}\varkappa,{{\cal V}_5})$,
%  \end{enumerate}
\end{ctext}
  Consequently, there are exactly $20$ isomorphism types of $\konftyp(15,4,20,3)$ skew perspectives with the skew
  determined by boolean completing.
\end{thm}

%% END OMIT

%% BEZ AUTOW

%%%%%%%%%%%%%%%%%%%%%%%%%%%%%%%%%%%%%%%%%%%%%%%%%%%%%%%%%%
%%%%%%%%%%%%%%%%%%%%%%%%%%%%%%%%%%%%%%%%%%%%%%%%%%%%%%%%%%
\fi %% END pominiecia kappy i END WYBIERANIA WARIANTÓW

%%%%%%%%%%%%%%%%%%%%%%%%%%%%%%%%%%%%%%%%%%%%%%%%%%%%%%%%%%
%%%%%%%%%%%%%%%%%%%%%%%%%%%%%%%%%%%%%%%%%%%%%%%%%%%%%%%%%%
%%%%%%%%%%%%%%%%%%%%%%%%%%%%%%%%%%%%%%%%%%%%%%%%%%%%%%%%%% BIBLIO
%%%%%%%%%%%%%%%%%%%%%%%%%%%%%%%%%%%%%%%%%%%%%%%%%%%%%%%%%%
%%%%%%%%%%%%%%%%%%%%%%%%%%%%%%%%%%%%%%%%%%%%%%%%%%%%%%%%%%

\let\bibtem\bibitem

\noindent Author's address:\newline
\ifcdn{Kamil Maszkowski, }\fi
Ma{\l}gorzata Pra{\.z}mowska, Krzysztof Pra{\.z}mowski\\
Institute of Mathematics, University of Bia{\l}ystok\\
ul. Cio{\l}kowskiego 1M\\
15-245 Bia{\l}ystok, Poland\\
e-mail: \ifcdn{\ttfamily{kamil33221@o2.pl}, }\fi {\ttfamily malgpraz@math.uwb.edu.pl},
{\ttfamily krzypraz@math.uwb.edu.pl}


\begin{thebibliography}{9}\itemsep-2pt
\bibitem{doliwa1}
  {\sc A. Doliwa},
  {\it The affine Weil group symmetry of Desargues maps and the non-commutative
  Hirota-Miwa system},
  Phys. Lett. A {\bf 375} (2011), 1219--1224.
\bibitem{doliwa2}
  {\sc A. Doliwa},
  {\it Desargues maps and the Hirota-Miwa equation},
  Proc. R. Soc. A {\bf 466} (2010), 1177--1200.
\bibitem{perspect}
  {\sc M. Pra{\.z}mowska},
  {\it Multiple perspectives and generalizations of the Desargues configuration},
  Demonstratio Math. {\bf 39} (2006), no. 4, {887--906}.
\bibtem{pascvebl}
  {\sc M. Pra{\.z}mowska, K. Pra{\.z}mowski},
  {\it Some generalization of Desargues and Veronese configurations},
  Serdica Math. J. {\bf 32} (2006), no {2--3}, {185--208}.
\bibitem{klik:binom}
  {\sc M. Pra{\.z}mowska, K. Pra{\.z}mowski},
  {\it Binomial partial Steiner triple systems containing complete graphs},
  Graphs Combin. {\bf 32}(2016), no. 5, {2079--2092}.
%% \bibitem{binkonf}
\bibitem{klik:VC}
  {\sc K. Petelczyc, M. Pra{\.z}mowska},
  {\it ${10}_{3}$-configurations and projective realizability of multiplied configurations},
  Des. Codes Cryptogr. {\bf 51}, no. 1 (2009), {45--54}. 
\bibitem{klin}
  {\sc M. Ch. Klin, R. P{\"o}schel, K. Rosenbaum},
  {\it Angewandte Algebra f{\"u}r Mathematiker und Informatiker},
  VEB Deutcher Verlag der Wissenschaften, Berlin 1988
\bibitem{skewgras}
  {\sc M. Pra\.zmowska},
  {\it On some regular multi-{V}eblen configurations, the geometry of combinatorial 
  quasi {G}rassmannians},
  Demonstratio Math. {\bf 42}(2009), no.1 {2}, {387--402}.
\bibitem{combver}
  {\sc M. Pra{\.z}mowska, K. Pra{\.z}mowski},
  {\it Combinatorial Veronese structures, their geometry, and problems of %
  embeddability},
  Results Math. {\bf 51} (2008), 275--308.
\bibitem{STP3K5}
  {\sc K. Petelczyc, M. Pra{\.z}mowska},
  {\it A complete classification of the $(15_4 20_3)$-configurations with 
  at least three {$K_5$}-graphs},
  Discrete Math. {\bf 338} (2016), no {7}, {1243--1251}.
\bibitem{pasz}
  {\sc A. C. H. Ling, C. J. Colbourn, M. J. Granell, T. S. Griggs},
  {\it Construction techniques for anti-Pasch Steiner triple systems},
  Jour. London Math. Soc. {\bf 61} (2000), no. 3, 641--657.
\bibitem{coxdes}
  {H. S. M. Coxeter},
  {\it Desargues configurations and their collineation groups},
  Math. Proc. Camb. Phil. Soc. {\bf 78}(1975), 227--246.
\bibitem{projmono1}
  {\sc H. S. M. Coxeter},
  {\sl Introduction to Geometry},
  John Wiley, 1989.
\bibitem{projmono2}
  {\sc R. Hartshorne}, 
  {\sl Foundations of projective geometry}, 
  Lecture Notes, Harvard University, 1967.
\bibitem{projchain}
  {\sc H. Karzel, H.-J. Kroll},
  {\it Perspectivities in Circle Geometries},
  [in]
  {\sl Geometry -- von Staudt's point of view}, 
  {P. Plaumann, K. Strambach} (Eds), D. Reidel Publ. Co., 1981,
  pp. 51--100.
\ifcdn\relax\else
\bibitem{maszko}
  {\sc K. Maszkowski, M. Pra{\.z}mowska, K. Pra{\.z}mowski},
  {\it Configurations representing a skew perspective},
  mimeographed.
\fi
\bibitem{projectiv}
  {\sc G. Pickert},
  {\it Projectivites in Projective Planes},
  [in]
  {\sl Geometry -- von Staudt's point of view}, 
  {P. Plaumann, K. Strambach} (Eds), D. Reidel Publ. Co., 1981,
  pp. 1--50.
\bibitem{mveb2proj}
  {\sc M. M. Pra{\.z}mowska},
  {\it On the existence of projective embeddings of multiveblen configurations},
  {Bull. Belg. Math. Soc. Simon-Stevin},
  {\bf 17}, (2010), no {2}, {1--15}.
\bibitem{hall}
  {\sc M. Hall Jr.},
  {\sl Combinatorial Theory},
  J. Wiley 1986.
\bibitem{bona}
  {\sc G. E. Andrews},
  {\sl The theory of Partitions},
  Cambride Univ. Press, 1998.
\bibitem{hilbert}
  {\sc D. Hilbert, S. Cohn-Vossen},
  {\sl Geometry and the Imagination},
  AMS Chelsea Publishing, 1999.
\end{thebibliography}
\end{document}